\documentclass[envcountsect,envcountsame,oribibl]{llncs}
\pdfoutput=1
\usepackage[utf8]{inputenc}
\usepackage{lmodern}
\usepackage[T1]{fontenc}
\usepackage[english]{babel}
\usepackage[numbers,sort&compress]{natbib}
\usepackage{microtype}
\usepackage{authblk}
\usepackage{amsmath}
\usepackage{amssymb}
\usepackage{hyperref}
\usepackage{paralist}
\usepackage[capitalize,noabbrev]{cleveref}
\usepackage{tikz}
\usetikzlibrary{matrix,calc}
\hypersetup{colorlinks,allcolors=black}
\pgfdeclarelayer{background}
\pgfsetlayers{background,main}
\numberwithin{equation}{section}
\numberwithin{figure}{section}

\spnewtheorem{observation}[theorem]{Observation}{\bfseries}{\rmfamily}
\spnewtheorem{algorithm}[theorem]{Algorithm}{\bfseries}{\rmfamily}
\spnewtheorem{construction}[theorem]{Construction}{\bfseries}{\rmfamily}
\spnewtheorem{bfproblem}[theorem]{Problem}{\bfseries}{\rmfamily}

\makeatletter
\def\NAT@spacechar{~}%
\makeatother

\crefname{construction}{Construction}{Constructions}
\crefname{observation}{Observation}{Observations}
\crefname{bfproblem}{Problem}{Problems}
\makeatletter
\providecommand*{\shuffle}{%
  \mathbin{\mathpalette\shuffle@{}}%
}
\newcommand*{\shuffle@}[2]{%
  \sbox0{$#1\vcenter{}$}%
  \kern .15\ht0 %
  \rlap{\vrule height .25\ht0 depth 0pt width 2.5\ht0}%
  \raise.1\ht0\hbox to 2.5\ht0{%
    \vrule height 1.75\ht0 depth -.1\ht0 width .17\ht0 %
    \hfill
    \vrule height 1.75\ht0 depth -.1\ht0 width .17\ht0 %
    \hfill
    \vrule height 1.75\ht0 depth -.1\ht0 width .17\ht0 %
  }%
  \kern .15\ht0 %
}
\makeatother

\newcommand{\fp}{fi\-xed-pa\-ra\-me\-ter}

\title{Precedence-constrained scheduling problems parameterized by partial order width}

\author{René van Bevern\inst{1}\fnmsep\thanks{Supported by project~16-31-60007 mol\_a\_dk of the Russian Foundation for Basic Research.}
\and Robert Bredereck\inst{2}
\and Laurent Bulteau\inst{3}
\and Christian~Komusiewicz\inst{4}\fnmsep\thanks{Supported by the DFG, project MAGZ (KO~3669/4-1).}
\and Nimrod Talmon\inst{5}\fnmsep\thanks{Supported by a postdoctoral fellowship from I-CORE ALGO.}
\and Gerhard J.\ Woeginger\inst{6}}

\institute{Novosibirsk State University, Novosibirsk, Russian Federation, \texttt{rvb@nsu.ru}
\and TU Berlin, Germany, \texttt{robert.bredereck@tu-berlin.de}
\and Université Paris-Est Marne-la-Vallée, France, \texttt{l.bulteau@gmail.com}
\and Friedrich-Schiller-Universität Jena, Germany, \texttt{christian.komusiewicz@uni-jena.de}
\and Weizmann Institute of Science,
Rehovot, Israel, \texttt{nimrodtalmon77@gmail.com}
\and TU Eindhoven, The Netherlands, \texttt{gwoegi@win.tue.nl}}

\usepackage[textsize=scriptsize]{todonotes}

\newcommand{\poly}{\text{poly}}
\newcommand{\interval}[2]{\ensuremath{\{#1,\dots,#2\}}} 

\newcommand{\chain}{\ell}
\newcommand{\ept}{\sigma}

\newcommand{\rtyp}{\ensuremath{\rho}}

\newcommand{\lag}{\ensuremath{\lambda}} %
\newcommand{\pwid}{\ensuremath{w}} %
\newcommand{\Cmax}{\ensuremath{C_{\max}}}
\newcommand{\pTwoPOnetwo}{P2|prec,\(p_j{\in}\{1,2\}\)|\(C_{\max}\)}
\newcommand{\pThreeSizeOnetwo}{P3|prec,\(p_j{=}1\),size\(_j{\in}\{1,2\}\)|\(C_{\max}\)}
\newcommand{\pTwoChains}{P2|chains|\(C_{\max}\)}
\newcommand{\pThreePOne}{P3|prec,\(p_j{=}1\)|\(C_{\max}\)}
\newcommand{\pPOne}{P|prec,\(p_j{=}1\)|\(C_{\max}\)}
\newcommand{\pPrec}{P|prec|\(C_{\max}\)}

\newcommand{\kshuffle}{\textsc{Shuffle Product}}
\newcommand{\domset}{\textsc{Dominating Set}}
\newcommand{\RCPSP}{\textsc{Resource-Constraint Project Scheduling}}
\newcommand{\partition}{\textsc{Partition}}

\begin{document}
\maketitle
\pagestyle{plain}

\begin{abstract}
  Negatively answering a question posed by \citeauthor{MW15} (Math.\ Program.~154(1--2):533--562), we show that \pTwoPOnetwo{}, the problem of finding a non-preemptive minimum-makespan schedule for precedence-constrained jobs of lengths~1 and~2 on two parallel identical machines, is W[2]-hard parameterized by the width of the partial order giving the precedence constraints.  To this end, we show that \kshuffle{}, the problem of deciding whether a given word can be obtained by interleaving the letters of \(k\)~other given words, is W[2]-hard parameterized by~\(k\), thus additionally answering a question posed by \citeauthor{RV13} (CSR~2013).  Finally, refining a geometric algorithm due to \citeauthor{Ser00} (Diskretn.\ Anal.\ Issled.\ Oper.~7(1):75--82), we show that the more general \RCPSP{} problem is fixed-parameter tractable parameterized by the partial order width combined with the maximum allowed difference between the earliest possible and factual starting time of a job.
\end{abstract}


\section{Introduction}
We study the parameterized complexity of the following NP-hard problem and various special cases~\citep{Ull75,SZ15} with respect to the width of the given partial order.
\begin{bfproblem}[Resource-constrained project scheduling (\textsc{RCPSP})]\label{prob:rcpsp}
  \begin{compactdesc}
  \item[\it Input:] A set~\(J\) of jobs, a partial order~\(\preceq\) on~\(J\), a set~\(R\) of renewable resources, for each resource~\(\rtyp\in R\) the available amount~\(R_\rtyp\), and for each~\(j\in J\) a processing time~\(p_j\in\mathbb N\) and the amount~\(r_{j\rtyp}\leq R_\rtyp\) of resource~\(\rtyp\in R\) that it consumes.
  \item[\it Find:] A \emph{schedule}~\((s_j)_{j\in J}\), that is, a \emph{starting time~\(s_j\in\mathbb N\)} of each job~\(j\), such that
    \begin{enumerate}
    \item\label{sched1} for \(i\prec j\), job~\(i\) finishes before job~\(j\) starts, that is, \(s_i+p_i\leq s_j\),
    \item\label{sched2} at any time~\(t\), at most \(R_\rtyp\)~units of each resource~\(\rtyp\) are used, that is, \(\sum_{j\in s(t)}r_{j\rtyp}\leq R_\rtyp\), where \(s(t):=\{j\in J\mid t\in[s_j,s_j+p_j)\}\), and
    \item\label{sched3} the maximum completion time~\(\Cmax:=\max_{j\in J}(s_j+p_j)\) is minimum.
    \end{enumerate}
  \end{compactdesc}
A schedule satisfying \eqref{sched1}--\eqref{sched2} is \emph{feasible}; a schedule satisfying~\eqref{sched1}--\eqref{sched3} is \emph{optimal}.
\end{bfproblem}
Intuitively, a schedule \((s_j)_{j\in J}\) processes each job~\(j\in J\) non-preemptively in the half-open real-valued interval~\([s_j,s_j+p_j)\), which costs \(r_{j\rtyp}\)~units of resource~\(\rtyp\) during that time.  After finishing, jobs free their resources for later jobs.  If there is only one resource and each job~\(j\) requires one unit of it, then RCPSP is equivalent to \pPrec{}, the NP-hard problem of non-preemptively scheduling precedence-constrained jobs on a given number~\(m\) of parallel identical machines to minimize the maximum completion time~\citep{Ull75}.

\looseness=-1 \citet{MW15} asked whether \pPrec{} is solvable in \(f(p_{\max},\pwid)\cdot\poly(n)\)~time, where \(p_{\max}\) is the maximum processing time, \(\pwid\) is the width of the given partial order~\(\preceq\), \(n\)~is the input size, and \(f\)~is a computable function independent of the input size.  In other words, the question is whether \pPrec{} is \emph{fixed-parameter tractable} parameterized by \(p_{\max}\) and \(\pwid\).  Motivated by this question, which we answer negatively, we strengthen hardness results for \pPrec{} and refine algorithms for RCPSP with small partial order width.

\medskip\noindent Due to space constraints, some details are deferred to an Appendix.

\paragraph{Stronger hardness results.}
We obtain new hardness results for the following special cases of~\pPrec{} (for basic definitions of parameterized complexity terminology, see the end of this section and recent textbooks~\citep{CFK+15,DF13}):
\begin{asparaenum}[(1)]
\item\label{hard1} \pTwoChains, the case with two machines and precedence constraints given by a disjoint union of total orders, remains weakly NP-hard for width~3.
\item\label{hard2} \pTwoPOnetwo, the case with two machines and processing times~1 and~2, is W[2]-hard parameterized by the partial order width~\(\pwid\).
\item\label{hard3} \pThreeSizeOnetwo, the case with three machines, unit processing times, but where each job may require one or two machines, is also W[2]-hard parameterized by the partial order width~\(\pwid\).
\end{asparaenum}

Towards showing \eqref{hard2} and \eqref{hard3}, we show that \kshuffle{}, the problem of deciding whether a given word can be obtained by interleaving the letters of \(k\)~other given words, is W[2]-hard parameterized by~\(k\).  This answers a question of \citet{RV13}.  We put these results into context in the following.

Result \eqref{hard1}~complements the fact that \pPrec{} with constant width~\(w\) is solvable in pseudo-polynomial time using dynamic programming~\citep{Ser00} and that \pTwoChains{} is \emph{strongly} NP-hard for unbounded width \citep{DLY91}.

Result \eqref{hard2} complements the NP-hardness
result for \pTwoPOnetwo{} due to \citet{Ull75} and the
W[2]-hardness result for \pPOne{} parameterized by the number~\(m\) 
machines due to~\citet{BF95}.  While not made explicit, one can observe that
\citeauthor{BF95}' reduction creates hard instances with~\(w=m+1\). This is
remarkable since \pPrec{} is trivially polynomial-time solvable
if~\(w\leq m\), and also since the result negatively answered
\citeauthor{MW15}'s question~\cite{MW15} twenty years before it was
posed.  Our result~\eqref{hard2}, however, gives a stronger negative
answer: unless W[2]\({}={}\)FPT, not even \pTwoPOnetwo{} allows for
the desired \(f(w)\cdot\poly(n)\)-time algorithm.

\paragraph{Refined algorithms.} \citet{Ser00} gave a geometric pseudo-polynomial-time algorithm for RCPSP with constant partial order width~\(\pwid\).  The degree of the polynomial depends on~\(\pwid\) and, by \eqref{hard1} above, the algorithm cannot be turned into a true polynomial-time algorithm unless P\({}={}\)NP even for constant~\(w\).
\looseness=-1 We refine this algorithm to solve RCPSP in \((2\lag+1)^\pwid\cdot 2^\pwid\cdot\poly(n)\)~time, where \(\lag\)~is the maximum allowed difference between earliest possible and factual starting time of a job.  The degree of the polynomial depends neither on~\(\pwid\) nor~\(\lag\) and is indeed a polynomial of the input size~\(n\).  This does not contradict \eqref{hard1} since the factor~\((2\lag+1)^\pwid\) might be superpolynomial in~\(n\).
We note that fixed-parameter tractability for \(\pwid\) or \(\lag\) alone is ruled out by~\eqref{hard2} and by \citet{LR78}, respectively.

\paragraph{Preliminaries.}  A reflexive, symmetric, and transitive relation~\(\preceq\) on a set~\(X\) is a \emph{partial order}.  We write \(x\prec y\) if \(x\preceq y\) and \(x\ne y\).  A subset \(X'\subseteq X\) is a \emph{chain} if \(\preceq\) is a total order on~\(X'\); it is an \emph{antichain} if the elements of~\(X'\) are mutually incomparable by~\(\preceq\).  The \emph{width} of~\(\preceq\) is the size of largest antichain in~\(X\).  A \emph{chain decomposition} of~\(X\) is a partition~\(X=X_1\uplus\dots\uplus X_k\) such that each~\(X_i\) is a chain.  %

\looseness=-1 Recently, the \emph{parameterized complexity} of scheduling problems attracted increased interest~\citep{Bev16}.  The idea %
is to accept exponential running times for solving NP-hard problems, but to restrict them to a small \emph{parameter} \citep{DF13,CFK+15}.  Instances~\((x,k)\) of a \emph{parameterized problem}~\(\Pi\subseteq\Sigma^*\times\mathbb N\) consist of an input~$x$ and a parameter~$k$.  A~parameterized problem~$\Pi$ is \emph{fixed-parameter tractable} %
if it is solvable in $f(k) \cdot \poly(|x|)$~time for some computable function~$f$.  Note that %
the degree of the polynomial must not depend on~$k$.  FPT is the class of \fp{} tractable parameterized problems.  There is a hierarchy of parameterized complexity classes FPT\({}\subseteq{}\)W[1]\({}\subseteq{}\)W[2]\({}\subseteq\dots\subseteq{}\)W[P], where all inclusions are conjectured to be strict.  A~parameterized problem~$\Pi_2$ is W[\(t\)]\emph{-hard} if there is a \emph{parameterized reduction} from each problem~\(\Pi_1\in{}\)W[\(t\)] to~\(\Pi_2\), that is, an algorithm that maps an instance~\((x,k)\) of~\(\Pi_1\) to an instance~$(x',k')$ of~\(\Pi_2\) in time~$f(k)\cdot\poly(|x|)$ such that $k'\leq g(k)$ and $(x,k)\in\Pi_1\Leftrightarrow(x',k')\in\Pi_2$, where \(f\)~and~\(g\) are arbitrary computable functions.  No W[\(t\)]-hard problem is fixed-parameter tractable unless FPT\({}={}\)W[\(t\)].

\section{Parallel identical machines and shuffle products}
This section presents our hardness results for special cases of \pPrec{}.
In \cref{sec:weakhardness}, we show weak NP-hardness of \pTwoChains{} for three chains.
In \cref{sec:kshuffle}, we show W[2]-hardness of \kshuffle{} as a stepping stone towards showing W[2]-hardness of \pThreeSizeOnetwo{} and \pTwoPOnetwo{} parameterized by the partial order width in \cref{sec:stronghardness}.

\subsection{Weak NP-hardness for two machines and three chains}
\label{sec:weakhardness}
\citet{DLY91} showed that \pTwoChains{} is strongly NP-hard.  We complement this result by the following theorem.
\begin{theorem}\normalfont\label{thm:weakhardness}
  \looseness=-1\pTwoChains{} is weakly NP-hard even for precedence constraints of width three, that is, consisting of three chains.
\end{theorem}

\begin{proof}[sketch]
\looseness=-1 We reduce from the weakly NP-hard \partition{} problem~\cite[SP12]{GJ79}:
Given a multiset of positive integers $A=\{a_1,\ldots,a_t\}$, decide whether there is a subset~$A' \subseteq A$ such that
$\sum_{a_i \in A'} a_i = \sum_{a_i \in A \setminus A'}a_i$.
Let $A=\{a_1,\ldots,a_t\}$ be a \partition{} instance. If $b:=\bigl(\sum_{a_i \in A} a_i\bigr)/2$ is not an integer, then we are facing a no-instance.  Otherwise, we construct a \pTwoChains{}~instance as follows.
Create three chains $J^0:=\{j^0_1\prec\dots\prec j^0_t\}$, $J^1:=\{j^1_1\prec\dots\prec j^1_{t+1}\}$,  and $J^2:=\{j^2_1\prec\dots\prec j^2_{t+1}\}$ of jobs.
For each~$i\in\interval1t$, job~$j^0_i$ gets processing time~$a_i$. The jobs in~\(J^1\cup J^2\) get processing time~$2b$ each.  This construction can be performed in polynomial time and one can show that the input \partition{} instance is a yes-instance of and only if the created \pTwoChains{}~instance allows for a schedule with makespan~\(T:=(2t+3)b\):  in such a schedule, each machine must perform exactly $t+1$~jobs from~$J^1 \cup J^2$ and has \(b\)~time for jobs from~\(J^0\).\qed
\end{proof}

\subsection{W[2]-hardness for Shuffle Product}\label{sec:kshuffle}

\looseness=-1 In this section, we show a W[2]-hardness result for \kshuffle{} that we transfer to %
\pTwoPOnetwo{} and \pThreeSizeOnetwo{}  in \cref{sec:stronghardness}.  We first formally introduce the problem (cf.\ \cref{fig:shuffle}).

\newcommand{\charA}{{\ensuremath{a}}}
\newcommand{\charB}{{\ensuremath{b}}}
\newcommand{\charC}{{\ensuremath{c}}}
\newcommand{\str}{word}
\newcommand{\Str}{Word}
\newcommand{\charCount}[2]{\lvert #1 \rvert_{#2}}
\newcommand{\charCountA}[1]{\charCount{#1}{\charA}}%
\newcommand{\charCountB}[1]{\charCount{#1}{\charB}}%

\begin{definition}[shuffle product]\normalfont\label{def:shuffprod}
  By \(s[i]\), we denote the \(i\)th letter in a word~\(s\).  A \str{}~$t$ is said to be in the \emph{shuffle product} of \str{}s~$s_1$ and~$s_2$, denoted by~$t\in s_1 \shuffle s_2$, if $t$~can be obtained by interleaving the letters of~$s_1$ and~$s_2$.  Formally, $t\in s_1 \shuffle s_2$ if there are increasing functions~$f_1\colon\interval1{|s_1|}\rightarrow\interval1{|t|}$ and~$f_2\colon\interval1{|s_2|}\rightarrow\interval1{|t|}$ mapping positions of~\(s_1\) and~\(s_2\) to positions of~\(t\) such that, for all~$i\in\interval1{|s_1|}$ and~$j\in\interval1{|s_2|}$, one has $t[f_1(i)]=s_1[i]$, $t[f_2(j)]=s_2[j]$, and $f_1(i)\neq f_2(j)$.
  This product is associative and commutative, which implies that the shuffle product of any set of \str{}s is well-defined.  
\begin{figure}[t]\centering
 \begin{tikzpicture}
 \matrix(M)[matrix of nodes,
        nodes={align=center,text width=4mm},
        column 1/.style={nodes={text width=7mm,align=right}}
    ]{
       $s_1=$ & \charA &    &  & \charC &  &  \charB & & &   &\charB \\
       $s_2=$ &  &    & \charB &  & \charB &  & \charC & && \\
       $s_3=$ &  & \charC   &  &  &  &  &              & \charA&\charB\\
       $t=$  & \charA & \charC   & \charB & \charC & \charB & \charB & \charC  & \charA&\charB&\charB\\       
    };
    \foreach \y [count=\i from 2] in {1,3,2,1,2,1,2,3,3,1} {
       \path[dashed,->](M-\y-\i) edge %
       (M-4-\i);
    }
 \end{tikzpicture}
 \vspace{-1em}
 \caption{\label{fig:shuffle}Illustration of a shuffle product: for \(s_1=\charA \charC \charB \charB\), \(s_2=\charB\charB\charC\), and \(s_3=\charC\charA\charB\), one has
  $t=\charA \charC \charB \charC \charB \charB \charC  \charA \charB\charB\in s_1\shuffle s_2\shuffle s_3$. Dashed arcs show how the letters of each $s_i$ map into $t$. }
\end{figure}
\end{definition}

\begin{bfproblem}[(Binary) Shuffle Product]
  \begin{compactdesc}
  \item[\it Input:] \Str{}s~$s_1,\dots,s_k$, and~$t$ over a (binary) alphabet~$\Sigma$.
  \item[\it Parameter:] $k$.
  \item[\it Question:] Is $t\in s_1 \shuffle s_2 \shuffle \dots \shuffle s_k$?
  \end{compactdesc}
\end{bfproblem}
\looseness=-1\textsc{Binary} \kshuffle{} is NP-hard for unbounded~\(k\)~\citep[Lemma~3.2]{WH84}, whereas \kshuffle{} is polynomial-time solvable for constant~$k$ using dynamic programming. \citet{RV13} asked about the parameterized complexity of \kshuffle{}.  We answer the question by the following theorem.

\begin{theorem}\label{thm:shuff}\normalfont
\textsc{Binary} \kshuffle{} is W[2]-hard.
\end{theorem}
Our proof uses a parameterized reduction from the W[2]-hard \domset{} problem~\citep{DF13,CFK+15} and is inspired by \citeauthor{BF95}'s proof that \pPOne{} is W[2]-hard parameterized by the number~\(m\) of machines~\citep{BF95}.
\begin{bfproblem}[Dominating Set]
  \begin{compactdesc}
  \item[\it Input:] A graph~\(G=(V,E)\) and a natural number~\(k\).
  \item[\it Parameter:] $k$.
  \item[\it Question:] Is there a size-$k$ \emph{dominating set}~\(D\), that is, \(V\subseteq N[D]\)?
  \end{compactdesc}
\end{bfproblem}
\noindent Herein, \(N[D]\)~is the set of vertices in~\(D\) and their neighbors. 
In order to describe the construction, we introduce some notation.
\begin{definition}\normalfont
  We denote the concatenation of words~\(s_1,\dots,s_k\) as \(\prod_{i=1}^ks_i:=s_1s_2\dots s_k\) and denote \(k\)~repetitions of a word~\(s\) by~\(s^k\).  The number of occurrences of a letter~$\charA$ in a word~$s$ is $\charCountA{s}$.
\end{definition}
\begin{construction}\label[construction]{constr:ds}
Given a \domset{} instance~\((G,k)\) with a graph~\(G=(V,E)\), we construct an instance of \textsc{Binary} \kshuffle{} with \(k+3\)~\str{}s over~\(\Sigma=\{a,b\}\) in polynomial time as follows.  The construction is illustrated in \cref{fig:reduction}.  Without loss of generality, assume that $V=\interval 1n$. 
\begin{align}
\text{For $u,v\in V$, let
 }\ell_{u,v}&:=
\begin{cases}
  1&\text{if $u=v$ or $\{u,v\}\in E$},\\
  2&\text{otherwise}.\label{defluv}
\end{cases}
\end{align}
Moreover, define two \str{}s
\begin{align*}
  A&:=\prod_{u=1}^n\prod_{v=1}^n \charA \charB^{\ell_{u,v}}&\text{and}&&
  B&:=\bigl((\charA^k \charB^{2k})^{n-1} \charA^k \charB^{2k-1}\bigr)^n.
\end{align*}
Finally, let $N:=2k(n-1)+1$ and output an instance of \kshuffle{} with the following \(k+3\)~\str{}s:
\begin{align*}
  s_i &:=A^N\text{\quad for each }i\in \interval1k,&  t&:= B^N (\charA^k \charB^{2k})^{n-1},\\
  s_{k+1}&:=\charA^{\charCountA{t}-\sum_{i=1}^k\charCountA{s_i}},\text{\quad and}&
  s_{k+2}&:=\charB^{\charCountB{t}-\sum_{i=1}^k\charCountB{s_i}}.
\end{align*}
\end{construction}
Note that \(A\)~is simply the \str{} that one obtains by concatenating the rows of the adjacency matrix of~\(G\) and replacing ones by~\(\charA\charB\) and zeroes by~\(\charA\charB\charB\).   
\begin{figure}[t]\centering
 \begin{tikzpicture}[inner sep=1pt] 
 \node[circle, draw] (a) at (0,0) {$v_1$}; 
 \node[circle, fill=black!20, draw] (b) at (1,0) {$v_2$}; 
 \node[circle, fill=black!20, draw] (c) at (0.5,0.8) {$v_3$}; 
 \draw[line width=1.5pt] (a)--(c);    
 \end{tikzpicture}\hfill
 \begin{tikzpicture} 
 \definecolor{myblue}{rgb}{0,0,0.8}
 \matrix(M)[matrix of nodes,
        nodes={align=center, inner sep=1.5pt,text width=3mm},
        row sep=1.2em,column sep=0em,
        column 1/.style={nodes={text width=7mm,align=right}}
    ]{
       $s_1=$ &&%
& \charA&  $\charB$& \charA&  $\charB\charB$& \charA&  $\charB$%
& \charA&  $\charB\charB$& \charA&  $\charB$& \charA&  $\charB\charB$%
& \charA&  $\charB$& \charA&  $\charB\charB$& \charA&  $\charB$\\ 
       $s_2=$ 
& \charA&  $\charB$& \charA&  $\charB\charB$& \charA&  $\charB$%
& \charA&  $\charB\charB$& \charA&  $\charB$& \charA&  $\charB\charB$%
& \charA&  $\charB$& \charA&  $\charB\charB$& \charA&  $\charB$\\ 
       $t=$ & $\charA^2$&$\charB^4$& $\charA^2$&$\charB^4$& $\charA^2$&$\charB^3$%
&$\charA^2$&$\charB^4$& $\charA^2$&$\charB^4$& $\charA^2$&$\charB^3$%
&$\charA^2$&$\charB^4$& $\charA^2$&$\charB^4$& $\charA^2$&$\charB^3$%
&$\charA^2$&$\charB^4$& $\charA^2$&$\charB^4$\\
    };
    \begin{pgfonlayer}{background}
      \foreach \i in {7,13,19} { \fill[fill=red!30]
        ($(M-1-\i.north west)+(-0.03,0.08)$) rectangle ($(M-3-\i.south east)+(0.03,-0.08)$); }
    \end{pgfonlayer}
    \foreach \i/\k/\j/\x  in {4/9/1/1,10/15/1/2,16/21/1/3,2/7/2/1,8/13/2/2,14/19/2/3} {
       
       \draw[draw=myblue, opacity=0.8] ($(M-\j-\i.north west)+(0.04,0.06)$) rectangle ($(M-\j-\k.south east)+(-0.04,-0.06)$);
        \node[scale=0.9,white,draw=myblue,thick, fill=myblue,inner sep=0.4mm, anchor=south west] at ($(M-\j-\i.north west)+(0.04,0.06)$){ {\boldmath$v_\x$}};
    }
 \end{tikzpicture}
 \caption{\label{fig:reduction}Left: A \domset{} instance with~\(k=2\) and a solution $\{v_2,v_3\}$ (the gray nodes). Right: The ``base pattern'' of the corresponding \kshuffle{} instance (only one repetition of~\(A\) in~\(s_1\) and~\(s_2\) and only one repetition of~\(B\) in~\(t\) is shown). Blocks of~\(s_1\) and~\(s_2\) are mapped into the blocks of~\(t\) displayed in the same column.  The horizontal (blue) rectangles reflect that each $s_i$ is built as the concatenation of the rows of the adjacency matrix, where zeroes are replaced by~$\charA\charB\charB$ and ones by~$\charA\charB$.  The amount of horizontal offset of each~$s_i$ corresponds to the selection of a vertex as dominator ($v_2$~for~$s_1$ and $v_3$ for~$s_2$). The dark columns (red) correspond to the short $\charB$-blocks of~$t$: they ensure that, in each row of the adjacency matrix, at least one selected vertex dominates the vertex corresponding to that row. The base pattern is repeated $N$~times to ensure that at least one occurrence of the pattern is mapped to~\(t\) without unwanted gaps. Additional \str{}s~$s_{k+1}$ and $s_{k+2}$ are added to match the remaining letters from~$t$.}
  \end{figure}
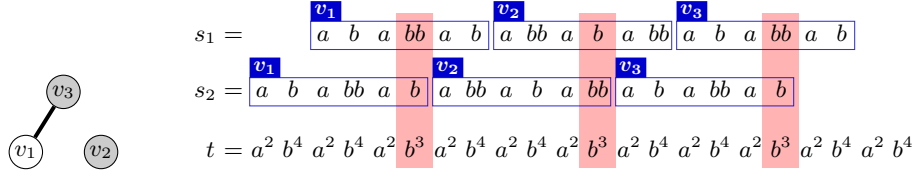

Before showing the correctness of \cref{constr:ds}, we make some basic observations about the \str{}s it creates, for which we introduce some terminology.

\begin{definition}[long and short blocks, positions]\normalfont\looseness=-1 A~\emph{block} in a \str{}~$s$ is a maximal consecutive sub\str{} using only one letter.  A~\emph{\charC{}-block} is a block containing only the letter~\charC{}. A block has \emph{position}~\(i\) in~\(s\) if it is the \(i\)th successive block in~\(s\).
 We call \charB-blocks of length~\(2k-1\) in~\(t\) \emph{short} and \charB-blocks of length~\(2k\) \emph{long}.  
\end{definition}

\begin{observation}\label[observation]{obs:ds}The \str{}s~\(s_1,\dots,s_k\) and~\(t\) created by \cref{constr:ds} from a \domset{} instance~\((G,k)\) have the following properties:
  \begin{enumerate}[(i)]
  \item Each $s_i$ for \(i\in\interval1k\) contains $2Nn^2$ blocks.
  \item\label{tblocks} The word~$t$ contains $2Nn^2+2(n-1)$ blocks.
  \item For \(i\in\interval1k\), all \charA-blocks in~$s_i$ have length~1. All \charA-blocks of~$t$ have length~$k$.
  \item\label{longpos} For $h\in\interval 1{Nn}$, the \charB-blocks at position~$2hn$ in~\(t\) are short. All other \charB-blocks in~\(t\) are long.
  \item\label{coradj} For each \(i\in\interval1k\), $p\in\interval0{N-1}$, and $u,v\in\interval1n$, the \charB-block at position $2pn^2 + 2n(u-1) + 2v$ in~\(s_i\) has length $\ell_{u,v}$: it corresponds to the entry in the \(u\)th row and \(v\)th column of the adjacency matrix of~\(G\).
  \end{enumerate}
\end{observation}
\looseness=-1 Since \cref{constr:ds} runs in polynomial time and the number of \str{}s in the created \kshuffle{} instance only depends on the size of the sought dominating set, for \cref{thm:shuff}, it remains to prove the following lemma.

\begin{lemma}\label{lem:W2Reduction}
Let \(s_1,\dots,s_{k+2}\) and~\(t\) be the \str{}s created by \cref{constr:ds} from a \domset{} instance~\((G,k)\).  Then \(G\)~has a dominating set of size~\(k\) if and only if \(t\in s_1\shuffle s_2\shuffle\dots\shuffle s_{k+2}\).
\end{lemma}

\begin{proof}
($\Rightarrow$) Assume first that $G=(V,E)$ has a dominating set~$D=\{d_1,\dots,d_k\}$. We describe~$t$ as a shuffle product of the \str{}s~$s_i$ as follows.  For each \(i\in\interval1k\), map all letters from the block at position~$x$ of~$s_i$ into block~$x + 2(n-d_i)$ of~$t$, that is, consecutive blocks of~$s_i$ are mapped into consecutive blocks of~$t$ with a small offset depending on~$d_i$.  So far, at most $k$~letters are mapped into each \charA-block of~\(t\) and at most $2k$~letters are mapped into each \charB-block of~$t$. Hence, all \charA-blocks and all long \charB-blocks of $t$ are long enough to accommodate all their designated letters. It remains to show that at most $2k-1$~letters are mapped into each short \charB-block~\(\beta\) of~$t$.  By \cref{obs:ds}\eqref{longpos}, \(\beta\)~is at position~$2hn$ for some~\(h\in\interval1{Nn}\).  Thus, there are $p\in\interval0{N-1}$ and $u\in\interval 1n$ such that \(2hn=2(pn+u)n=2pn^2+2un\).  For each $s_i$, the block~$\alpha_i$ of~$s_i$ mapped into~$\beta$ has position~$(2pn^2 +2un) - 2(n-d_i)=2pn^2 +2(u-1)n +2d_i$.  Hence, $\alpha_i$~has length~$\ell_{u,d_i}$ by \cref{obs:ds}\eqref{coradj}. Since $D$~is a dominating set, it contains a vertex~\(d_{i^*}\) such that $d_{i^*}=u$ or $\{d_{i^*}, u\}\in E$.  Thus, by \eqref{defluv}, $\alpha_{i^*}$~has length~$\ell_{u,d_{i^*}}=1$.  Overall, at most $k$ \charB-blocks of $\{s_1,\dots,s_k\}$ are mapped into~$\beta$.  We have shown that at least one of them, namely~\(\alpha_{i^*}\), has length one.  Since the others have length at most two, at most $2k-1$~letters are mapped into block~\(\beta\).

\looseness=-1 We have seen a mapping of the \str{}s~\(s_i\) with \(i\in\interval1k\) to~\(t\).  Thus, we have \(\charCountA{t}\geq\sum_{i=1}^k\charCountA{s_i}\) and \(\charCountB{t}\geq\sum_{i=1}^k\charCountB{s_i}\) and the words~\(s_{k+1}\) and~\(s_{k+2}\) are well-defined. It remains to map~$s_{k+1}$ and~$s_{k+2}$ to~\(t\).  Since \(s_{k+1}\)~consists only of \charA{} and \(s_{k+2}\) only of \charB{}, we only have to check that $t$~contains as many letters with letters~$\charA$ or~$\charB$ as all \str{}s~$s_i$ together, which is true by the definition of~\(s_{k+1}\) and~\(s_{k+2}\).
We conclude that $t\in s_1\shuffle s_2\shuffle\dots\shuffle s_{k+2}$ if \(G\)~has a dominating set of size~\(k\).

($\Leftarrow$) Assume that $t\in s_1\shuffle s_2\shuffle\dots\shuffle s_{k+2}$.  We show that \(G\)~has a dominating set of size~\(k\).  To this end, for \(i\in\interval1k\), let $y_i(x)$~be the position of the block in~$t$ into which the last letter of the block at position~\(x\) of~\(s_i\) is mapped and let $\delta_i(x)=y_i(x)-x$.  We will see that, intuitively, one can think of \(\delta_i(x)\) as the shift of the \(x\)th block of~\(s_i\) in~\(t\).  To show that \(G\)~has a dominating set of size~\(k\), we use the following two facts about~\(\delta_i\), which we will prove afterwards.
\begin{enumerate}[(i)]
\item\label{deltaeins} For $i\in\interval1k$ and $x\in\interval1{2Nn^2}$, one has $\delta_i(x)\in\interval0{2(n-1)}$.
\item\label{deltazwei} There is a $p\in\interval0{N-1}$ such that, for all $i\in\interval1k$, $\delta_i$ is constant over the interval~$I_p=\interval{2pn^2+1}{2(p+1)n^2+1}$. 
\end{enumerate}
We now focus on a $p\in\interval0{N-1}$ as in~\eqref{deltazwei} and write~$\delta_i$ for the value~\(\delta_i(x)\) taken for all~\(x\in I_p\).  We show that \(D:=\{d_i=n-\delta_i/2\mid k\in\interval1k\}\) is a dominating set of size~\(k\) for~\(G\), that is, we show \(D\subseteq V\) and \(V\subseteq N[D]\).

To this end, consider a vertex~$u\in V$ and the block~$\beta$ of~$t$ at position~$2pn^2 + 2un=2hn$ for \(h=pn+u\in\interval1{Nn}\).  By \cref{obs:ds}\eqref{longpos}, $\beta$~is a short \charB-block.  For any \(i\in\interval1k\), let $\alpha_i$~be the block at position~$2pn^2 + 2un - \delta_i$ in~$s_i$.  Because of~\eqref{deltaeins}, this position is in~\(I_p\).  By definition of~\(\delta_i\), the last letter of~$\alpha_i$ is mapped into~$\beta$. Thus, $\alpha_i$ is a \charB-block.  Note that this implies that $\delta_i$~is even since \charA-blocks and \charB-blocks are alternating in~\(t\) and~\(s_i\).  Moreover, by~\eqref{deltaeins}, $d_i=n-\delta_i/2\in\interval1n=V$.  It follows that \(D\subseteq V\).  We show that \(u\in N[D]\).  To this end, note that the \charA-block in~\(s_i\) at position \(2pn^2+2un-\delta_i-1\in I_p\) directly preceding~$\alpha_i$ is mapped into the \charA-block of~\(t\) at position~\(2pn^2+2un-1\) directly preceding~$\beta$.  Thus, all letters of~$\alpha_i$ are mapped into~$\beta$ and one has
\begin{align}\label{eq:boundOnSize}
 \sum_{i=1}^k |\alpha_i|&\leq |\beta|.
 \end{align} 
By \cref{obs:ds}\eqref{coradj}, $\alpha_i$~has length $\ell_{u,(n-\delta_i/2)}=\ell_{u,d_i}$.  Since $\beta$~is a short \charB-block, it has length~$2k-1$. From~\eqref{eq:boundOnSize}, we get $\sum_{i=1}^k \ell_{u,d_i} \leq 2k-1$.  Thus, there is some~$i^*\in\interval1k$ with $\ell_{u,d_{i^*}}=1$. By~\eqref{defluv}, that means $d_{i^*}=u$ or \(\{u,d_{i^*}\}\) is an edge in~$G$.  Hence, \(u\in N[D]\) and \(D\)~is a dominating set of size~\(k\) for of~$G$.

\looseness=-1 It remains to prove \eqref{deltaeins} and \eqref{deltazwei}.  For \eqref{deltaeins}, note that $y_i(1)\geq 1$ and $y_i(x+1)\geq y_i(x) +1$.  Hence, $\delta_i$~is non-decreasing with all values being non-negative.  Furthermore, for $x=2Nn^2$, $y_i(x)\leq 2Nn^2+2(n-1)$ since \(t\)~has only so many blocks by \cref{obs:ds}\eqref{tblocks}.  Thus, the maximum possible value of~$\delta_i$ is~$2(n-1)$.  Towards~\eqref{deltazwei}, we say that a value of~$p\in\interval0{N-1}$ is \emph{bad for~$i$} if $\delta_i$~is not constant over~$I_p$. For such a~$p$, one has $\delta_i(2pn^2+1)< \delta_i(2(p+1)n^2+1)$.  Hence, there can be at most $2(n-1)$~values of~$p$ that are bad for~$i$. Overall, there are at most $2k(n-1)<N$~values of~$p$ that are bad for some~$i\in\interval1k$.  Thus, at least one value is not bad for any~$i$.  For this value of~$p$, every~$\delta_i$ is constant over the interval~$I_p$.\qed
\end{proof}

\subsection{W[2]-hardness of scheduling problems parameterized by width} \label{sec:stronghardness}
In the previous section, we showed W[2]-hardness of \kshuffle{}.  We now transfer this result to scheduling problems on parallel identical machines.

\begin{theorem}\label{thm:stronghardness}\normalfont The following two problems are W[2]-hard parameterized by the width of the partial order giving the precedence constraints.

  \begin{inparaenum}[(i)]
  \item\label{hardone} \pTwoPOnetwo{},\qquad
  \item\label{hardtwo} \pThreeSizeOnetwo{}.
  \end{inparaenum}
\end{theorem}
\noindent We prove \eqref{hardone} using the following parameterized reduction from \kshuffle{} with \(k+1\)~words to \pTwoPOnetwo{} with \(k+2\)~chains.
\begin{construction}\label{constr:stronghardness}
Let $(s_1, \ldots, s_k, t)$~be a \kshuffle{} instance over the alphabet~\(\Sigma=\{1,2\}\). Assume that $\charCount t1=\sum_{i=1}^k\charCount {s_i}1$ and $\charCount t2=\sum_{i=1}^k\charCount {s_i}2$ (otherwise, it is a no-instance).  We create an instance of \pTwoPOnetwo{}:
  \begin{asparaenum}[(1)]
  \item For each $i\in \interval1k$, create a chain of \emph{worker} jobs~\(j_{i1}\prec j_{i2}\prec\dots\prec j_{i|s_i|}\), where $j_{i,x}$~has  length~$s_i[x]$.%
  \item \looseness=-1 For each $x\in\interval1{|t|}$, create three \emph{floor} jobs
 $z_{x,1}$, $z_{x,2}$, $z_{x,3}$ with $z_{x,1}\prec z_{x,2}$ and $z_{x,1}\prec z_{x,3}$, where $z_{x,1}$ has length~$t[x]$, and $z_{x,2}$ and $z_{x,3}$~have length~1.  If $x<|t|$, then also add the precedence constraints~$z_{x,2}\prec z_{x+1,1}$ and $z_{x,3}\prec z_{x+1,1}$.
  \end{asparaenum}
\end{construction}

\noindent %
Observe that $\{z_{x,1},z_{x,2} \mid 1\leq x\leq |t|\}$ is chain.  Thus, the makespan of any schedule is at least $T:=\sum_{x=1}^{|t|}(t[x]+1)$. For $x\in\interval1n$, let $\tau(x):=\sum_{i=1}^{x-1}(t[x]+1)$. 
\begin{observation}\label{obs:stronghardness}
A schedule with makespan~$T$ must schedule job~$z_{x,1}$ at time~$\tau(x)$, and jobs~$z_{x,2}$ and~$z_{x,3}$ at time~$\tau(x)+t[x]$. Thus, 
 for $x\in\interval1{|t|}$, both machines are used by floor jobs from $\tau(x)+t[x]$ to $\tau(x)+t[x]+1$ and one machine is free of floor jobs between~$\tau(x)$ and~\(\tau(x)+t[x]\) for \(t[x]\)~time units.  We call these \emph{available time slots}.
\end{observation}
\looseness=-1\noindent \cref{constr:stronghardness} runs in polynomial time.  Moreover, from \(k+1\)~input words, it creates instances of width~\(k+2\): there are \(k\)~chains of worker jobs and the floor decomposes into two chains $\{z_{x,1},z_{x,2} \mid 1\leq x\leq |t|\}$ and $\{z_{x,1},z_{x,3} \mid 1\leq x\leq |t|\}$.  To prove \cref{thm:stronghardness}\eqref{hardone}, one can thus show that \(t\in s_1\shuffle\dots\shuffle s_{k}\) if and only if the created \pTwoPOnetwo{} instance allows for a schedule of makespan~\(T\). By \cref{obs:stronghardness}, any such schedule has available time slots of lengths corresponding to the letters in~\(t\), each of which can accommodate a worker job corresponding to a letter of~\(s_1,\dots,s_k\).  The precedence constraints ensure that these worker jobs get placed into the time slots corresponding to letters of~\(t\) in increasing order.

The proof of \cref{thm:stronghardness}\eqref{hardtwo} works analogously: one simply replaces worker jobs of length two by worker jobs of length one that require two machines and modifies the floor jobs so that they do not create time slots of length one or two, but so that each created time slot is available on only one or on two machines.  To achieve this, the construction uses three machines.

\section{Resource-Constrained Project Scheduling}
In \cref{sec:stronghardness}, we have seen that \pThreeSizeOnetwo{} is W[2]-hard parameterized by the partial order width.  It follows that also RCPSP  (cf.\ \cref{prob:rcpsp}) is W[2]-hard for this parameter, even if the number of resources and the maximal resource usage are bounded by two and all jobs have unit processing times.
In this section, we additionally consider the \emph{lag} parameter:
\begin{definition}[earliest possible starting time, lag]\normalfont\label{def:ept}
  Let \(J_0\subseteq J\)~be the jobs that are minimal elements in the partial order~\(\preceq\).
  The \emph{earliest possible starting time~\(\ept_j\)} is 0 for a job~\(j\in J_0\) and, inductively, \(\max_{i\prec j}(\ept_i+p_i)\) for a job~\(j\in J\setminus J_0\).  %
  The \emph{lag} of a feasible schedule~\((s_j)_{j\in J}\) is \(\lag:=\max_{j\in J} s_j-\ept_j\).
\end{definition}
\noindent \citeauthor{LR78}'s NP-hardness proof for \pPOne{}~\citep{LR78} shows that it is even NP-hard to decide whether there is a schedule of makespan at most three and lag at most one.  Thus, the lag~\(\lag\) alone cannot lead to a fixed-parameter algorithm for RCPSP, just as the width~\(\pwid\) alone cannot.  We show a fixed-parameter algorithm for the parameter~\(\lag+w\).

\begin{theorem}\normalfont\label{thm:rcpsp}
  An optimal schedule with lag at most~\(\lag\) for RCPSP is computable in \((2\lag+1)^\pwid\cdot 2^\pwid\cdot\poly(n)\)~time if it exists, where \(w\)~is the partial order width.
\end{theorem}
Our algorithm is a refinement of \citeauthor{Ser00}'s pseudo-polynomial-time algorithm for RCPSP{} with constant width~\citep{Ser00}, which is based on graphical optimization methods introduced by \citet{Ake56} and \citet{HN63} for hand-optimizing \textsc{Job Shop} schedules for two jobs.  We provide a concise translation of \citeauthor{Ser00}'s algorithm in \cref{sec:pseudofpt} before we prove \cref{thm:rcpsp} in \cref{sec:realfpt}.

\subsection{Geometric interpretation of RCPSP}\label{sec:pseudofpt}
\noindent Given an RCPSP instance with precedence constraints~\(\preceq\) of width~\(\pwid\), by Dilworth's theorem, we can decompose our set~\(J\) of jobs into \(\pwid\)~pairwise disjoint chains. More specifically, these chains are efficiently computable~\citep{FRS03}.  For~\(\chain\in\{1,\dots,\pwid\}\), denote the jobs in chain~\(\chain\) by a sequence~\((j_{\chain k})_{k=1}^{n_\chain}\) such that \(j_{\chain k}\prec j_{\chain k+1}\) and let
\begin{itemize}[ \(L_\chain^i:=\)]
\item[\(L_\chain^i:=\)]\(\sum_{k=1}^{i}p_{j_{\chain k}}\) be the sum of processing times of the first \(i\)~jobs on chain~\(\chain\),
\item[\(L_\chain:=\)]\(L_\chain^{n_\chain}\) be the sum of processing times of all jobs on chain~\(\chain\).
\end{itemize}
Let~\(\vec 0:=(0,\dots,0)\in\mathbb R^\pwid\) and \(\vec L:=(L_1,\dots,L_\pwid)\).  Each point in the \(\pwid\)-dimensional orthotope~\(X:=\{\vec x\in\mathbb R^{\pwid}\mid \vec0\leq \vec x \leq\vec L\}\) describes a \emph{state} as follows.
\begin{definition}[running, completed, feasibility]\normalfont\label{def:state}
  Let \(\vec x=(x_1,\allowbreak\dots,\allowbreak x_\pwid)\in X\).  For each chain~\(\chain\in\interval1w\), if \(x_\chain\in[L_{\chain}^{i-1},L_{\chain}^{i})\), then the jobs~\((j_{\chain k})_{k=1}^{i-1}\) of chain~\(\chain\) are \emph{completed} and job~\(j_{\chain i}\) has been \emph{processed for \(x_\chain - L_{\chain}^{i-1}\)~time}.  We call job~\(j_{\chain i}\) \emph{running} if~\(L_\chain^{i-1}<x_\chain<L_\chain^i\).  We denote by
  \begin{itemize}[ \(C(\vec x)\subseteq J\)]
  \item[\(J(\vec x)\subseteq J\)] the set of jobs running in state~\(\vec x\) and by
  \item[\(C(\vec x)\subseteq J\)] the set of jobs completed in state~\(\vec x\).
  \end{itemize}
  A point~\(\vec x\in X\) is \emph{feasible} if it holds that both
  \begin{enumerate}[({IF}1)]
  \item\label{infeas1} the jobs~\(J(\vec x)\) comply with resource constraints, that is, \(\sum_{j\in J(\vec x)}r_{j\rtyp}\leq R_\rtyp\) for each resource~\(\rtyp\in R\), and
  \item\label{infeas2} if there are two jobs~\(i\prec j\) such that~\(j\in J(\vec x)\), then \(i\in C(\vec x)\).
  \end{enumerate}
\end{definition}
\looseness=-1Note that %
points~\(\vec x\in X\) may indeed violate~(IF\ref{infeas2}): there are not only precedence constraints between jobs on one chain, but also between jobs on different chains.

Each feasible schedule now yields a path of feasible points in the orthotope~\(X\) from the point~\(\vec 0\), where no job has started, to the point~\(\vec L\), where all jobs are completed.  Each such path consists of (linear) segments of the form~\([\vec x,\vec x+t\vec\delta]\) for some~\(\vec\delta=(\delta_1,\dots,\delta_\pwid)\in\{0,1\}^\pwid\), which corresponds to running exactly the jobs on the chains~\(\chain\) with~\(\delta_\chain=1\) for \(t\)~units of time.  Since all processing times and starting times are integers (cf.\ \cref{prob:rcpsp}), we can assume~\(t\in\mathbb N\).
\begin{definition}[feasibility of segments and their lengths]\normalfont
  The \emph{length} of a segment~\([\vec x,\vec x+t\vec\delta]\) is~\(t\).  The \emph{length} of a path is the sum of the lengths of its segments.  A segment~\([\vec x,\vec x+t\vec\delta]\) is \emph{feasible} if it contains only feasible points and interrupts no jobs; that is, if there is a job~\(j\in J(\vec x)\) on chain~\(\chain\), then \(\delta_\chain=1\).
\end{definition}
There is now a one-to-one correspondence between feasible schedules and paths from~\(\vec 0\) to~\(\vec L\) consisting only of feasible segments and between the shortest of these paths and optimal schedules.  This leads to the following algorithm.
\begin{algorithm}[\citet{Ser00}]\label{alg:servakh}
Compute a shortest feasible path from~\(\vec 0\) to~\(\vec L\) using dynamic programming:  for each feasible point~\(\vec x\in X\cap\mathbb N^\pwid\) in lexicographically increasing order, compute the length~\(P(\vec x)\) of a shortest feasible path from~\(\vec 0\) to~\(\vec x\) using the recurrence relation
\begin{align}
  P(\vec 0)&=0, &P(\vec x)=\min_{\vec\delta\in\Delta_{\vec x}}P(\vec x-\vec \delta)+1\text{\quad for feasible }\vec x\in X\cap\mathbb N^\pwid\setminus\{\vec 0\},\label{eq:recur}
\end{align}
  where \(\Delta_{\vec x}\) is the set of vectors~\(\vec \delta\in\{0,1\}^\pwid\) such that segment~\([\vec x-\vec\delta,\vec x]\) is feasible.
\end{algorithm}
To compute~\(P(\vec L)\), one thus iterates over at most 
\(\prod_{\chain=1}^\pwid (L_\chain+1)\)
points~\(\vec x\in X\cap\mathbb N^\pwid\), for each of them over \(2^\pwid\)~vectors \(\vec\delta\in\{0,1\}^\pwid\), and, for each, decides whether~\([\vec x-\vec\delta,\vec x]\)~is feasible.  Since the set of running jobs is the same for all interior points of the segment, it is enough to check the feasibility of its end points and one interior point, which can be done in polynomial time.  Thus, the algorithm runs in
\(
  \prod_{\chain=1}^\pwid (L_\chain+1)\cdot2^\pwid\cdot\poly(n)~\text{time},
\)
which is pseudo-polynomial for constant~\(\pwid\).
\subsection{Fixed-parameter algorithm for arbitrary processing times}\label{sec:realfpt}
\looseness=-1 The bottleneck of \cref{alg:servakh} is that it searches for a shortest path from~\(\vec 0\) to~\(\vec L\) in the whole orthotope~\(X\).  For the case where we are only accepting schedules of maximum lag~\(\lag\), we will shrink the search space significantly: we show that we only have to search for paths within a tight corridor around the path corresponding to the schedule~\((\sigma_j)_{j\in J}\) that starts jobs at the earliest possible time.

\begin{definition}[point at time \boldmath\(t\) on a path]\normalfont
  Let \(p\)~be the path from~\(\vec 0\) to~\(\vec L\) corresponding to a not necessarily feasible schedule \((s_j)_{j\in J}\) that, however, respects precedence constraints.  Let  \(t\geq 0\) and \(T\)~be the length of~\(p\).

  Then, 
\(\vec p(t)\) is the endpoint of the subpath of length~\(t\) of~\(p\) starting in~\(\vec 0\) for \(t\leq T\), and \(\vec p(t):=\vec L\) for \(t>T\).
\end{definition}
Since the definition requires \((s_j)_{j\in J}\)~to respect precedence constraints, \(\vec p(t)\)~determines the state (cf.\ \cref{def:state}) at time~\(t\) according to schedule~\((s_j)_{j\in J}\).

\begin{definition}[\boldmath\(\lag\)-corridored]\normalfont\label{def:corridor}
Let \(p\)~be the path corresponding to the schedule~\((\sigma_j)_{j\in J}\) that starts jobs at the earliest possible time (cf.\ \cref{def:ept}).  
 \[\Gamma_\lag(t):=\{\vec x\in X\mid \vec p(t)-\vec \lag\leq\vec x\leq\vec p(t)\},\text{\quad where } \vec\lag=(\lag,\dots,\lag)\in\mathbb N^\pwid.   \]
 We call a path \(q\) \emph{\(\lag\)-corridored} if \(\vec q(t)\in\Gamma_\lag(t)\) for all~\(t\geq 0\).
\end{definition}

\noindent Note that points on the path~\(p\) in \cref{def:corridor} may violate \cref{def:state}(IF\ref{infeas1}), but not (IF\ref{infeas2}).  One can show the following relation between \(\lag\)-corridored paths and schedules of lag~\(\lag\).

\begin{lemma}\normalfont\label{lem:corridor}
  A feasible schedule~\((s_j)_{j\in J}\) has lag at most~\(\lag\) if and only if its corresponding path~\(q\) is \(\lag\)-corridored.
\end{lemma}

\noindent\looseness=-1\cref{lem:corridor} allows us to compute a shortest feasible path from~\(\vec 0\) to~\(\vec L\) using only points in~\(\Gamma_\lag(t)\) for some~\(t\).  Herein, we will exploit the following condition for checking whether a path segment can be part of a \(\lag\)-corridored path.

\begin{lemma}\label{lem:pseudoconvex}\normalfont
  Let \([\vec x,\vec x+t\vec\delta]\)  for \(\vec\delta\in\{0,1\}^\pwid\). %
  If \(\vec x\in\Gamma_\lag(t_0)\) and \(\vec x+t\vec\delta\in\Gamma_\lag(t_0+t)\) for some~\(t_0\geq 0\), then \(\vec x+\tau\vec\delta\in\Gamma_\lag(t_0+\tau)\) for all \(0\leq\tau\leq t\).
\end{lemma}

\begin{proof}
  Let \(p\)~be the path corresponding to schedule~\((\ept_j)_{j\in J}\) as in \cref{def:corridor} and let \(\vec \delta=(\delta_1,\dots,\delta_\pwid)\in\{0,1\}^\pwid\). %
For any \(\tau\in [0,t]\), consider 
\begin{align*}
\vec x^\tau=(x_1^\tau,\dots,x^\tau_\pwid):=\vec x+\tau\vec\delta&&
\text{and}&&
\vec y^\tau=(y_1^\tau,\dots,y_\chain^\tau):=\vec p(t_0+\tau).
\end{align*}
  By the prerequisites of the lemma, we have \(\vec y^0-\vec\lag\leq \vec x^0\leq \vec y^0\) and  \(\vec y^t-\vec\lag\leq \vec x^t\leq \vec y^t\).  We show \(\vec y^\tau-\vec\lag\leq\vec x^\tau\leq\vec y^\tau\)  for any \(\tau\in [0,t]\).

 We start with \(\vec x^\tau\leq\vec y^\tau\).  For the sake of contradiction, assume that there is some chain~\(\chain\) and a~\(\tau\in[0,t]\) such that \(x_\chain^\tau>y_\chain^\tau\).  Then, \(x_\chain^\tau>y_\chain^\tau\geq y_\chain^0\geq x_\chain^0\).  It follows that \(\delta_\chain=1\), which contradicts  \(\vec x^t\leq \vec y^t\) because, then,
\begin{align*}
x_\chain^t=x^0_\chain+t=x^0_\chain+\tau+(t-\tau)=x_\chain^\tau+(t-\tau)>y_\chain^\tau+(t-\tau)\geq y_\chain^{\tau+(t-\tau)}= y_\chain^t.
\end{align*}
\looseness=-1 Now, we show \(\vec y^\tau-\vec\lag\leq\vec x^\tau\).  Consider some chain~\(\chain\).  If \(\delta_\chain=1\), then we have \( y_\chain^\tau-\lag\leq y_\chain^0+\tau-\lag\leq x_\chain^0+\tau=x_\ell^\tau \) and we are fine.
If \(\delta_\chain=0\) and there is %
a \(\tau\in[0,t]\) such that \(y_\chain^\tau-\lag>x_\chain^\tau\), then \( y_\chain^t-\lag\geq y_\chain^\tau-\lag>x_\chain^\tau=x_\chain^t\), contradicting \(\vec y^t-\vec\lag\leq\vec x^t\).\qed
\end{proof}
\noindent
We can now prove the following result by computing recurrence~\eqref{eq:recur} for each of the \((\lag+1)^\pwid\)~feasible points~\(\vec x\in \Gamma_\lag(t)\cap\mathbb Z^\pwid\) for all~\(t\in\interval0L\).
\begin{proposition}\normalfont
  An optimal schedule of lag at most~\(\lag\) for RCPSP if it exists is computable in~\((\lag+1)^\pwid\cdot 2^\pwid\cdot\poly(L)\)~time, where \(L\)~is the sum of all processing times and \(\pwid\)~is the partial order width.
\end{proposition}
\noindent\looseness=-1 However, note that this is a fixed-parameter algorithm only for polynomial processing times, which is why we skip the proof and go on towards proving \cref{thm:rcpsp}---a  fixed-parameter algorithm that works for arbitrarily large processing times.  To this end, we prove that all maximal segments of a path corresponding to a schedule with lag at most~\(\lag\) start and end in one of \(2\cdot|J|\)~hypercubes with edge length~\(2\lag+1\).%

\begin{lemma}\label{lem:bigcubes}\normalfont
  Let \(q\)~be the path of a feasible schedule~\((s_j)_{j\in S}\) of lag at most~\(\lag\) and let \(t_2\leq t_1\leq t_2+\lag\).  Then, \(\vec q(t_1)\in  \Gamma_{2\lag}(t_2+\lag)\) (cf.\ \cref{def:corridor}).
\end{lemma}
\begin{proof}
  Consider the schedule~\((\sigma_j)_{j\in J}\) that starts each job at the earliest possible time and its path~\(p\).  
Our aim is to show \[\vec p(t_2+\lag)-2\vec\lag\leq\vec q(t_1)\leq\vec p(t_2+\lag),\] where \(\vec\lag=(\lag,\dots,\lag)\in\mathbb N^\pwid\).
By \cref{lem:corridor}, \(q\)~is \(\lag\)-corridored. Thus,
\begin{align*}
  \vec p(t_1)-\vec\lag&\leq \vec q(t_1)\leq\vec p(t_1)&\text{and}&&
  \vec p(t_2+\lag)-\vec\lag&\leq\vec q(t_2+\lag)\leq\vec p(t_2+\lag).
\end{align*}
From this, one easily gets \(\vec q(t_1)\leq\vec p(t_1)\leq\vec p(t_2+\lag)\).  Moreover, one has
\[
\vec p(t_2+\lag)-2\vec\lag\leq \vec q(t_2+\lag)-\vec\lag\leq\vec q(t_2)+\vec\lag-\vec\lag=\vec q(t_2)\leq\vec q(t_1).\quad\qed
\]
\end{proof}

\begin{lemma}\label{lem:boxes}\normalfont
  Let \(q\)~be the path of a feasible schedule~\((s_j)_{j\in S}\) of lag at most~\(\lag\) and let \([\vec x,\vec x+t\vec\delta]\)~be a maximal segment of~\(q\) such that the set~\(J(\vec x+\tau\vec\delta)\) of running jobs (cf.\ \cref{def:state}) is the same for all \(\tau\in(0,t)\). Then,
\[\{\vec x,\vec x+t\vec\delta\}\subseteq\Gamma:=\bigcup_{j\in J}\Gamma_{2\lag}(\ept_j+\lag)\cup\bigcup_{j\in J}\Gamma_{2\lag}(\ept_j+p_j+\lag),\]
where \((\sigma_j)_{j\in J}\)~is the schedule that starts each job at the earliest possible time. %
\end{lemma}
\begin{proof}
Let \(t_0\)~be chosen arbitrarily such that \(\vec q(t_0)\in\{\vec x,\vec x+t\vec\delta\}\).  By maximality of the segment, some job~\(j\in J\) is starting or ending at time~\(t_0\), that is, \(t_0=s_j\) or \(t_0=s_j+p_j\).  Then, \(\{\vec x,\vec x+t\vec\delta\}\subseteq\Gamma_{2\lag}(\sigma_j+\lag)\cup\Gamma_{2\lag}(\sigma_j+p_j+\lag)\) follows from \(\ept_j\leq s_j\leq\ept_j+\lag\) and \cref{lem:bigcubes}.\qed
\end{proof}

\noindent We are now ready to show a fixed-parameter algorithm for RCPSP parameterized by length and maximum lag.  That is, we prove \cref{thm:rcpsp}.
\begin{proof}[of \cref{thm:rcpsp}]
  We compute the shortest feasible \(\lag\)-corridored path from the state~\(\vec 0\), were no job has started, to the state~\(\vec L\), where all jobs have been completed (cf.\ \cref{lem:corridor}).  We use dynamic programming similarly to \cref{alg:servakh}.  By \cref{lem:boxes}, it is enough to consider those paths whose segments start and end in~\(\Gamma\). %
  Thus, for each~\(\vec x\in\Gamma\cap\mathbb N^\pwid\) in lexicographically increasing order, we compute the length~\(P(\vec x)\) of a shortest \(\lag\)-corridored path from~\(\vec 0\) to~\(\vec x\) with segments starting and ending in~\(\Gamma\).  To this end, for an~\(\vec x\in\Gamma\cap\mathbb N^\pwid\), let \(\Delta_{\vec x}\) be the set of vectors~\(\vec\delta\in\{0,1\}^\pwid\) such that,
  \begin{enumerate}[(i)]
  \item\label{dx1} there is a smallest integer~\(t_{\vec\delta}\geq 1\) such that \(x-t_{\vec\delta}\cdot\vec\delta\in\Gamma\) and such that
  \item\label{dx2} the segment~\([\vec x-t_{\vec\delta}\cdot\vec\delta,\vec x]\) is feasible.
  \end{enumerate}
  Then, \(P(\vec 0)=0\) and, for feasible~\(x\in\Gamma\cap\mathbb N^\pwid\setminus\{\vec 0\}\), one has
  \begin{align*}
    P(\vec x)={}&\min\{P(\vec x-t_{\vec\delta}\cdot\vec \delta)+t_{\vec\delta}\mid \vec\delta\in\Delta_{\vec x}\text{ and }\vec x\in\Gamma_\lag(P(\vec x-t_{\vec\delta}\cdot\vec \delta)+t_{\vec\delta})\},
  \end{align*}
where \(\min\emptyset=\infty\) and the last condition on \(\vec x\)~uses \cref{lem:pseudoconvex} to ensure that we are indeed computing the length~\(P(\vec x)\) of a \emph{\(\lag\)-corridored} path  (cf.\ \cref{def:corridor}) to~\(\vec x\): by induction, we know that \(P(\vec x-t_{\vec\delta}\cdot\vec\delta)\) is the length of a shortest \emph{\(\lag\)}-corridored path to~\(\vec x-t_{\vec\delta}\cdot\vec\delta\), and thus \(\vec x-t_{\vec\delta}\cdot\vec\delta\in\Gamma_\lag(P(\vec x-t_{\vec\delta}\cdot\vec\delta))\).

  We have to discuss how to check \eqref{dx1} and \eqref{dx2}.  One can check \eqref{dx2} in polynomial time since it is enough to check feasibility at the end points and one interior point of the segment since the set of jobs running at the interior points of \([\vec x-t_{\vec\delta}\cdot\vec\delta,\vec x]\) does not change: otherwise, since jobs are started or finished only at integer times, there is a maximal subsegment \([\vec x,\vec x-t\cdot\vec\delta]\) with \(t\leq t_{\vec\delta}-1\) where the set of running jobs does not change.  Then \(\vec x-t\cdot\vec\delta\in\Gamma\)~by \cref{lem:boxes}, contradicting the minimality of~\(t_{\vec\delta}\).
  
\looseness=-1  Towards \eqref{dx1}, we search for the minimum \(t_{\vec\delta}\geq 1\) such that \(\vec x-t_{\vec\delta}\cdot\vec\delta\in\Gamma\).  Consider the schedule \((\ept_j)_{j\in J}\) that schedules each job at the earliest possible time (cf.\ \cref{def:ept}).  It is computable in polynomial time.  By \cref{lem:boxes}, we search for the minimum~\(t_{\vec\delta}\geq 1\) such that \(\vec x-t_{\vec\delta}\cdot\vec\delta\in\Gamma_{2\lag}(\ept_j+\lag)\) or \(\vec x-t_{\vec\delta}\cdot\vec\delta\in\Gamma_{2\lag}(\ept_j+p_j+\lag)\) for some job~\(j\in J\).  That is, by \cref{def:corridor}, for each job~\(j\), we find the minimum~\(t_j\geq 1\) that solves a 
  system of linear inequalities of the form
\(\vec y-2\vec\lag\leq \vec x-t_j\cdot\vec\delta\leq\vec y\), where \(\vec\delta=(\delta_1,\dots,\delta_\pwid)\in\{0,1\}^\pwid\).  Writing \(\vec y=(y_1,\dots,y_\pwid)\) and \(\vec x=(x_1,\dots,x_\pwid)\), either \(t_j=\max(\{1\}\cup\{x_\chain-y_\chain\mid\delta_\chain=1\})\) is the minimum such~\(t_j\) or there is no solution for job~\(j\).  Note that \(t_j\)~is an integer since \(\vec x\) and \(\vec y\) are integer vectors.  Thus, \(t_{\vec\delta}=\min_{j\in J}t_j\)~is computable in polynomial time.

\looseness=-1 We conclude that we process each~\(\vec x\in\Gamma\cap\mathbb N^\pwid\) in \(2^\pwid\cdot\poly(n)\)~time. %
Moreover, \(\Gamma\)~contains at most \(2\cdot|J|\cdot(2\lag+1)^\pwid\)~integer points since each job~\(j\in J\) contributes at most \((2\lag+1)^\pwid\)~points in \(\Gamma_{2\lag}(\ept_j+\lag)\) and at most \((2\lag+1)^\pwid\)~points in~\(\Gamma_{2\lag}(\ept_j+p_j+\lag)\).  A total running time of \((2\lag+1)^\pwid\cdot 2^\pwid\cdot\poly(n)\) follows.\qed
\end{proof}

\section{Conclusion}

\looseness=-1 Our algorithm for RCPSP shows, in particular, that \pThreePOne{} is fixed-parameter tractable parameterized by the partial order width~\(w\) and allowed lag~\(\lag\).  Since the NP-hardness of this problem is a long-standing open question~\citep[OPEN8]{GJ79}, it would be surprising to show W[1]-hardness of this problem for \emph{any} parameter: this would exclude polynomial-time solvability unless FPT\({}={}\)W[1].  Thus, it makes sense to search for a fixed-parameter algorithm for \pThreePOne{} parameterized by~\(w\), whereas we showed that already \pTwoPOnetwo{} and \pThreeSizeOnetwo{} are W[2]-hard parameterized by~\(w\).

\newpage
\paragraph{Acknowledgments.} 
The authors are thankful to Sergey Sevastyanov for pointing out the work of \citet{Ake56} and \citet{Ser00}.  This research was initiated at the annual research retreat of the algorithms and complexity group of TU Berlin, April 3--9, 2016, Krölpa, %
Germany. 

\bibliographystyle{splncsnat}
\bibliography{p-prec-cmax}

\newcommand{\noopsort}[1]{}
\begin{thebibliography}{16}
\providecommand{\natexlab}[1]{#1}
\providecommand{\url}[1]{\texttt{#1}}
\providecommand{\urlprefix}{}

\bibitem[{Akers(1956)}]{Ake56}
Akers, Jr., S.B.: A graphical approach to production scheduling problems.
\newblock Oper.\ Res. 4(2), 244--245 (1956)

\bibitem[{{\noopsort{Bevern}van Bevern}(2016)}]{Bev16}
{\noopsort{Bevern}van Bevern}, R.: {FPT} in operations research: Opportunities
  and challenges.
\newblock Parameterized Complexity News 12(1) (2016), to appear

\bibitem[{Bodlaender and Fellows(1995)}]{BF95}
Bodlaender, H.L., Fellows, M.R.: W[2]-hardness of precedence constrained
  $k$-processor scheduling.
\newblock Oper.\ Res.\ Lett. 18(2), 93--97 (1995)

\bibitem[{Cygan et~al.(2015)Cygan, Fomin, Kowalik, Lokshtanov, Marx, Pilipczuk,
  Pilipczuk, and Saurabh}]{CFK+15}
Cygan, M., Fomin, F.V., Kowalik, L., Lokshtanov, D., Marx, D., Pilipczuk, M.,
  Pilipczuk, M., Saurabh, S.: Parameterized Algorithms.
\newblock Springer (2015)

\bibitem[{Downey and Fellows(2013)}]{DF13}
Downey, R.G., Fellows, M.R.: Fundamentals of Parameterized Complexity.
\newblock Springer (2013)

\bibitem[{Du et~al.(1991)Du, Leung, and Young}]{DLY91}
Du, J., Leung, J.Y.T., Young, G.H.: Scheduling chain-structured tasks to
  minimize makespan and mean flow time.
\newblock Inform.\ Comput. 92(2), 219--236 (1991)

\bibitem[{Felsner et~al.(2003)Felsner, Raghavan, and Spinrad}]{FRS03}
Felsner, S., Raghavan, V., Spinrad, J.: Recognition algorithms for orders of
  small width and graphs of small {Dilworth} number.
\newblock Order 20(4), 351--364 (2003)

\bibitem[{Garey and Johnson(1979)}]{GJ79}
Garey, M.R., Johnson, D.S.: Computers and Intractability: A Guide to the Theory
  of {NP}-Completeness.
\newblock Freeman (1979)

\bibitem[{Hardgrave and Nemhauser(1963)}]{HN63}
Hardgrave, W.W., Nemhauser, G.L.: A geometric model and a graphical algorithm
  for a sequencing problem.
\newblock Oper.\ Res. 11(6), 889--900 (1963)

\bibitem[{Lenstra and {Rinnooy Kan}(1978)}]{LR78}
Lenstra, J.K., {Rinnooy Kan}, A.H.G.: Complexity of scheduling under precedence
  constraints.
\newblock Operations Research 26(1), 22--35 (1978)

\bibitem[{Mnich and Wiese(2015)}]{MW15}
Mnich, M., Wiese, A.: Scheduling and fixed-parameter tractability.
\newblock Math.\ Program. 154(1-2), 533--562 (2015)

\bibitem[{Rizzi and Vialette(2013)}]{RV13}
Rizzi, R., Vialette, S.: On recognizing words that are squares for the shuffle
  product.
\newblock In: Proc.\ 8th CSR. LNCS, vol. 7913, pp. 235--245. Springer (2013)

\bibitem[{Schwindt and Zimmermann(2015)}]{SZ15}
Schwindt, C., Zimmermann, J. (eds.): Handbook on Project Management and
  Scheduling, vol.~1.
\newblock Springer (2015)

\bibitem[{Servakh(2000)}]{Ser00}
Servakh, V.V.: Effektivno razreshimy sluchaj zadachi kalendarnogo
  pla\-ni\-ro\-va\-ni\-ya s vozobnovimymi resursami.
\newblock Diskretn.\ Anal.\ Issled.\ Oper. 7(1), 75--82 (2000)

\bibitem[{Ullman(1975)}]{Ull75}
Ullman, J.: {NP}-complete scheduling problems.
\newblock J.~Comput.\ Syst.\ Sci. 10(3), 384--393 (1975)

\bibitem[{Warmuth and Haussler(1984)}]{WH84}
Warmuth, M.K., Haussler, D.: On the complexity of iterated shuffle.
\newblock J.~Comput.\ Syst.\ Sci. 28(3), 345 -- 358 (1984)

\end{thebibliography}

\newpage\appendix
\section{Appendix: Omitted proofs}

\subsection{Proofs for \cref{sec:weakhardness}}\label{apx:weakhardness}

{
\renewcommand{\thetheorem}{\ref{thm:weakhardness}}
\begin{theorem}\normalfont
  \pTwoChains{} is weakly NP-hard even for precedence constraints consisting of three chains.
\end{theorem}
}
\begin{proof}
\looseness=-1 We reduce from the weakly NP-hard \partition{} problem~\cite[SP12]{GJ79}:
Given a multiset of positive integers $A=\{a_1,\ldots,a_t\}$, decide whether there is a subset~$A' \subseteq A$ such that
$\sum_{a_i \in A'} a_i = \sum_{a_i \in A \setminus A'}a_i$.
Let $A=\{a_1,\ldots,a_t\}$ be a \partition{} instance. If $b:=\bigl(\sum_{a_i \in A} a_i\bigr)/2$ is not an integer, then we are facing a no-instance.  Otherwise, we construct a \pTwoChains{}~instance as follows.
Create three chains of jobs $J^0:=\{j^0_1\prec\dots\prec j^0_t\}$, $J^1:=\{j^1_1\prec\dots\prec j^1_{t+1}\}$,  and $J^2:=\{j^2_1\prec\dots\prec j^2_{t+1}\}$.
The jobs~$j^0_i$ with~$i\in\interval1t$ have processing time~$a_i$. The jobs~$j^\ell_i$ with \(\ell\in\{1,2\}\) and $i\in\interval1{t+1}$ have processing time~$2b$.  This completes the construction, which can be performed in polynomial time.  We show that the input \partition{} instance is a yes-instance of and only if the created \pTwoChains{}~instance allows for a schedule with makespan~\(T:=(2t+3)b\).

\((\Leftarrow)\) Assume that our constructed \pTwoChains{} instance has a schedule of makespan~\(T\).  Since we have two chains~\(J^1\) and~\(J^2\) each containing $t+1$~\emph{long jobs} with processing time~$2b$, each machine must perform exactly $t+1$~long jobs from~$J^1 \cup J^2$ and may perform additional \emph{short jobs} with processing times at most~$b$ from~$J^0$.  Let $A'$~be the set of elements in~$A$ corresponding to the jobs from~$J^0$ processed by the first machine.  Then, $A\setminus A'$ corresponds to the jobs from~$J^0$ that are performed by the second machine.  Since the makespan is~$T=(2t+3)b$ and the long jobs already need $(2t+2)b$ time units on each machine, it holds that $\sum_{a_i \in A'} a_i = \sum_{a_i \in A \setminus A'}a_i=b$.  Thus, $A'$~is a solution for our \partition{} instance. 

\((\Rightarrow)\) Let~$A'\subseteq A$ with $\sum_{a_i \in A'} a_i = \sum_{a_i \in A \setminus A'}a_i=b$ be a solution for \partition{} and let $J'^0\subseteq J^0$~denote the set of jobs corresponding to the elements in $A'\subseteq A$.  We construct a schedule of makespan~\(T\) as follows.  First, ignoring the jobs in~$J^0$, schedule each jobs in~$J^1\cup J^2$ at the earliest possible time, that is, the starting time of each job equals the sum of processing times of all preceding jobs in its chain.  So far, the maximum completion time is $(2t+2)b$ and each chain is processed by one machine.  Next, we modify this schedule by ``inserting'' the jobs from~$J^0$ in between two already scheduled jobs.  Herein, \emph{inserting} job~$j$ between job~$j',j''\in J^z$ for $z \in \{1,2\}$ means to set the starting time of~$j$ to the starting time of~$j''$ and to increase the starting times of~$j''$ and of all its successors in~$J^z$ by the processing time of~$j$.  We insert all jobs from~$J_0$ according to the precedence constraints on~\(J_0\):  we insert job~$j^0_i$ between jobs~$j^1_i$ and~$j^1_{i+1}$ if $a_i \in A'$ and between job~$j^2_i$ and job~$j^2_{i+1}$ if~$a_i \in A \setminus A'$.  The way we defined the insertion operation ensures that the schedule can still realized by two machines.  Since $\sum_{a_i \in A'} a_i = \sum_{a_i \in A \setminus A'}a_i=b$, it further holds that the starting time of every job was increased by at most~$b$ and, hence, the constructed schedule has makespan at most~\(T=(2t+3)b\) (the latest executed jobs are still job~$j^1_{t+1}$ and job~$j^2_{t+1}$).  It remains to check that the precedence constraints are fulfilled: constraints between jobs from~$J^1$ or from~$J^2$ remain fulfilled since we do not change their relative execution order.  The constraints between jobs~$j^0_{i}$ and~$j^0_{i'}$ for $1\le i < i' \le t$ are fulfilled because, in our constructed schedule, the starting time of job~$j^0_{i}$ is in the interval~$[i\cdot2b,i\cdot2b+b)$, the starting time of job~$j^0_{i'}$ is in the interval~$[i'\cdot2b,i'\cdot2b+b)$, and these interval do not intersect.\qed
\end{proof}

\subsection{Proofs of \cref{sec:stronghardness}}\label{apx:stronghardness}
\begin{lemma}\normalfont
A \kshuffle{} instance~$(t, s_1, \ldots, s_k)$ is a yes-instance if and only if the \pTwoPOnetwo{} instance created by \cref{constr:stronghardness} allows for a schedule of makespan~\(T:=\sum_{x=1}^{|t|}(t[x]+1)\).
\end{lemma}
\begin{proof}
($\Rightarrow$) Consider $k$ functions $f_1,\ldots,f_k$ mapping the letters of $s_1,\ldots, s_k$ to the letters of $t$ as required by \cref{def:shuffprod}.
We use the schedule described by \cref{obs:stronghardness} for the floor jobs, and, for $i\in \interval1k$ and $x\in\interval1{|s_i|}$, we schedule the worker job~$j_{i,x}$ to time~$\tau(f_i(x))$.
Note that the precedence constraints of the worker jobs are satisfied since the functions~$f_i$ are strictly increasing and the difference between two consecutive values of~$\tau$ is at least~$2$ (which is the maximal length of a job). Moreover, for each $y\in\interval1{|t|}$ there is exactly one~$i$ such that~$y=f_i(x)$ for some~$x$.  Hence, a worker job~$j_{i,x}$ can use the available time slot at~$\tau(y)$ without any other worker job occupying it. Finally, this worker job~$j_{i,x}$ needs time~$s_i[x]=t[f_i(x)]=t[y]$, which is exactly the length of the available time slot at~\(\tau(y)\).

($\Leftarrow$) Consider a scheduling with makespan~$T$, and let $n_1=\charCount{t}{1}$ and $n_2=\charCount{t}{2}$.  We construct functions $f_1,\ldots,f_k$ mapping the letters of $s_1,\ldots, s_k$ to the letters of $t$ as required by \cref{def:shuffprod}.  From \cref{obs:stronghardness}, we know that only time slots of lengths one and two are available for worker jobs.  %
Hence, each job~$j_{i,x}$ with length~$s_i[x]=2$ is scheduled to a time~$\tau(y)$ for some $y\in\interval1{|t|}$ with~$t[y]=2$.  We put $f_i(x):=y$.  Since the number of worker jobs of length two is $\sum_{i=1}^k\charCount {s_i}2= n_2$, all available time slots of length two are used by worker jobs of length two.  Thus, jobs of length one must use pairwise distinct available time slots of length one.  Thus, each job~$j_{i,x}$ of length~$s_i[x]=1$ is scheduled to a time~$\tau(y)$ for some~$y\in\interval1{|t|}$ with $t[y]=1$. Put $f_i(x):=y$. Clearly, each of the constructed functions~$f_i$ is total on $\interval1{|s_i|}$.  Moreover, no two functions share a value since the execution times of worker jobs do not intersect.  Finally, each function is increasing: due to the precedence constraints, job~$j_{i,x}$ is scheduled before~$j_{i,x+1}$ and thus, $\tau(f_i(x))<\tau(f_i(x+1))$ and $f_i(x)<f_i(x+1)$ since $\tau$ is increasing. Finally, note that $t[f_i(x)]=s_i[x]$ by construction, hence functions~$f_i$ define a valid mapping of the letters of~$s_i$ into~$t$.\qed
\end{proof}

\subsection{Proofs of \cref{sec:realfpt}}\label{apx:realfpt}

{
\renewcommand{\thelemma}{\ref{lem:corridor}}
\begin{lemma}\normalfont
  A feasible schedule~\((s_j)_{j\in J}\) has lag at most~\(\lag\) if and only if its corresponding path~\(q\) is \(\lag\)-corridored.
\end{lemma}
}
\begin{proof}
  \looseness=-1 (\(\Rightarrow)\) Consider a point~\(\vec q(t)\) on~\(q\) and the corresponding point~\(p(t)\) on the path~\(p\) corresponding to schedule~\((\ept_j)_{j\in J}\) (cf.\ \cref{def:corridor}).  Then one has \(\vec q(t)\leq\vec p(t)\) since, at time~\(t\), schedule~\((s_j)_{j\in J}\) cannot have processed any chain for more time than schedule~\((\ept_j)_{j\in J}\).  Moreover, one has \(\vec q(t)\geq\vec p(t)-\vec\lag\) since, at time~\(t\), a chain~\(\chain\) that has been processed for \(x_\chain\)~time by schedule~\((\ept_j)_{j\in J}\) has been processed for at least~\(x_\chain-\lag\) time by schedule~\((s_j)_{j\in J}\).  Thus, \(\vec q(t)\in \Gamma_\lag(t)\).

\looseness=-1 (\(\Leftarrow)\)  We show that \(s_j-\ept_j\leq\lag\) for an arbitrary job~\(j\).  Since \(\vec q(s_j)\in\Gamma_\lag(s_j)\), one has \(\vec p(s_j)-\vec\lag\leq \vec q(s_j)\leq\vec p(s_j)\).  In particular \(\vec q(s_j)\leq\vec p(s_j)+\vec\lag\).  In state~\(\vec q(s_j)\), job~\(j\) has not been processed for any time yet.  It follows that it has been processed for at most \(\lag\)~time in state~\(\vec p(s_j)\).  Since this is the state of schedule~\((\ept_j)_{j\in J}\) at time~\(s_j\), we get \(\ept_j\geq s_j-\lag\), that is, \(s_j-\ept_j\leq\lag\).\qed
\end{proof}

\end{document}